\documentclass[11pt]{article}
\usepackage[dvips]{graphicx}
\usepackage{epsfig}
\usepackage{makeidx}

\usepackage{amsthm}
\usepackage{amscd}
\usepackage{amsmath}
\usepackage{amsfonts}
\usepackage{amssymb}
\usepackage{multirow}
\usepackage[comma]{natbib}

\newtheorem{theorem}{Theorem}

\newtheorem{lemma}{Lemma}
\newtheorem{cor}{Corolary}

\def\E{{\mathbb E}}
\def\P{{\mathbb P}}

\def\Z{{\mathbb Z}}

\def\MIR{{\mbox{MIR}}}
\def\d{{\mbox{d}}}

	\makeatletter
	\renewcommand\@biblabel[1]{#1.}
	\makeatother

\begin{document}

\title{Information theoretic interpretation of frequency domain connectivity  measures}
\date{22nd Nov. 2010}
\author{Daniel Yasumasa Takahashi, Luiz Antonio Baccal\'a, Koichi Sameshima}
\maketitle

\begin{abstract}
To provide adequate multivariate measures of information flow between neural structures, 
modified expressions of Partial Directed Coherence (PDC) and Directed Transfer Function (DTF), two popular multivariate 
connectivity measures employed in neuroscience, are introduced and their formal relationship to
mutual information rates are proved.
\end{abstract}

\section{Introduction}

Over the last decade Neuroscience has  been witnessing an important paradigm shift thanks to the fast advancement of multichannel data acquisition technology. This process has been marked by the growing realization that the brain's inner workings can only be grasped through a detailed description of how brain areas interact functionally in a scenario that has come to be generally referred as the study of brain connectivity  and which stands in sharp contrast to former
longstanding efforts mostly directed at merely identifying which brain areas were involved in specific functions.

As such, many techniques have been proposed to address this problem, specially because
of the need to process and make sense of many simultaneously acquired brain activity signals,  \citep{Kaminski91,Sommer2003,Astolfi2007}. Among the available methods, we introduced
and developed the idea of\ \textit{partial directed coherence} (PDC) \citep{baccalapbr,BaccaSame01} which consists of a means of dissecting the frequency domain relationship between pairs of signals from among a set of $K\geq 2$ simultaneously observed time series.

The main characteristic of PDC is that it decomposes the interaction of each pair of time series within the set
into directional components while deducting the possibly shrouding effect of the remaining $K-2$ series.
It has, for instance, been possible to show that PDC is related to the notion of Gran\-ger causality which corresponds to the ability of pinpointing the level of attainable improvement in predicting a time series $x_i(n)$ when the past of another time series $x_j(n)$ is known ($i\neq j$) \citep{Granger69}. 

In fact, multivariate Granger causality tests as described in \cite{Lutkepohl93} map directly onto statistical tests for PDC nullity. Like Granger causality, and as opposed to ordinary coherence \citep{Priestley81}, PDC is a directional quantity; this fact lead to the idea of 'directed' connectivity that allows one to expressly test for the presence of feedback and to the idea that PDC is somehow associated with the direction of information flow.

The appeal of associating PDC with information flow has been strong; we have used it ourselves \citep{baccalapbr,BaccaSame01}. Yet this suggestion has until now remained vague and to some extent almost apocryphal. The aim of this paper is to correct this state of affairs by making the relationship between PDC and information flow at once formally explicit and precise.

On a par with PDC, is the no less important notion of directed transfer function (DTF) by \cite{Kaminski91}, whose information theoretic interpretation is also addressed here.

By providing further details and full proofs, this paper expands on our previous publication \citep{BuenosAires2010} and is organized as follows: in Sec. \ref{sec:background} we provide some explicit information theoretic background leaving the main result to Sec. \ref{sec:main_result} followed by illustrations and comments in Sec. \ref{sec:illustrations} and \ref{sec:discussion} respectively. Detailed proofs are covered in the Appendix.

\section{BACKGROUND}
\label{sec:background}

The relationship between two discrete time stochastic processes $x=\left\{ x(k) \right\}_{k \in \mathbb{Z}}$ and $y=\left\{y(k) \right\}_{k \in \mathbb{Z}}$ is assessed via their mutual information rate $\MIR(x,y)$ by comparing their joint probability density with the product of their marginals:
{\small
\begin{align}
\label{eq:mutual_info}
&\MIR(x,y)= \notag \\
&\lim_{m \to \infty}\frac{1}{m+1} \E\left[\log \frac{\d\P(x(1), \ldots, x(m),y(1),\ldots, y(m))}{\d\P(x(1), \ldots,x(m))\d\P(y(1),\ldots,y(m))}\right]
\end{align}
}

\noindent
where $\E\left[\cdot \right]$ is the expectation with respect to the joint measure of $x$ and $y$ and where $\d\P$ denotes the appropriate probability density. 
An immediate consequence of \eqref{eq:mutual_info} is that independence between $x$ and $y$ implies $\MIR$ nullity.

The main classic result for jointly Gaussian stationary processes, due to \cite{Gelfand59},
relates \eqref{eq:mutual_info} to the coherence between the processes via
\begin{equation}
\label{eq:gy_coherence}
\MIR(x,y)=-\frac{1}{4\pi}\int^{\pi}_{-\pi}\log (1-|\mbox{C}_{xy}(\omega)|^2)d\omega,
\end{equation}
where the coherence in \eqref{eq:gy_coherence} is given by
\begin{equation}
\mbox{C}_{xy}(\omega)=\frac{S_{xy}(\omega)}{\sqrt{S_{xx}(\omega)S_{yy}(\omega)}},
\end{equation}  
with $S_{xx}(\omega)$ and $S_{yy}(\omega)$ standing for the autospectra and $S_{xy}(\omega)$ for the cross-spectrum, respectively.

The important consequence of this result is that the integrand in \eqref{eq:gy_coherence} may be interpreted as the
frequency decomposition of $\MIR(x,y)$.

In view of this result, the following questions arise: Does a similar result hold for PDC? How and in what sense? 

Before addressing these problems, consider the zero mean wide sense stationary vector process $\mathbf{x}(n)=\left[x_1(n) \dots x_K(n) \right]^T$ representable by multivariate autoregressive model
\begin{equation}
\label{eq:model_law}
\mathbf{x}(n)=\sum_{l=1}^{+\infty}\mathbf{A}(l)\mathbf{x}(n-l)+\mathbf{w}(n),
\end{equation}
where $\mathbf{w}(n)= \left[w_1(n) \dots w_K(n) \right]^T$ stand for zero mean wide sense
stationary innovation processes with positive definite covariance matrix $\boldsymbol{\Sigma}_\mathbf{w}=\E\left[\mathbf{w}(n)\mathbf{w}^T(n)\right]$. 

A sufficient condition for the existence of representation \eqref{eq:model_law} is that the spectral density matrix associated with the process $\left\{\mathbf{x}(n)\right\}_{n \in \Z}$
be uniformly bounded from below and above and be invertible at all frequencies \citep{hannan70}.  From the coefficients $a_{ij}(l)$ of $\mathbf{A}(l)$ we may write
\begin{equation}
\label{eq:acoef}
\bar{A}_{ij}(\omega)=\left\{ 
\begin{array}{l}
1-\sum\limits_{l=1}^{+\infty}a_{ij}(l)e^{-\mathbf{j}\omega l},\;\text{if}\;\;i=j \\ 
-\sum\limits_{l=1}^{+\infty}a_{ij}(l)e^{-\mathbf{j}\omega l},\;\text{otherwise}
\end{array}
\right.  
\end{equation}
where $\mathbf{j}=\sqrt{-1}$ for $\omega \in [-\pi,\pi)$. 

Also let
${\bf \bar{a}}_j(\omega)=\left[\bar{A}_{1j}(\omega) \dots  \bar{A}_{Kj}(\omega) \right]^T$ and
consider the quantity, henceforth termed \textit{information} PDC ($\iota$PDC) from $j$ to $i$,
\begin{equation}
\label{eq:full_PDC}
\iota\pi_{ij}(\omega)= 
\dfrac{\bar{A}_{ij}(\omega)\sigma_{ii}^{-1/2}}{\sqrt{{\bf \bar{a}}_j^H(\omega)%
\boldsymbol{\Sigma}_\mathbf{w}^{-1}{\bf \bar{a}}_j(\omega)}},  
\end{equation}
where $\sigma_{ii}=\E\left[w^2_i(n)\right]$ and which
simplifies to the originally defined PDC when $\boldsymbol{\Sigma}_\mathbf{w}$
equals the identity matrix. Note also that the generalized PDC ($g$PDC) from \cite{Baccala07}
is obtained if $\boldsymbol{\Sigma}_\mathbf{w}$ is a diagonal matrix whose elements are not necessarily the same.

Before stating the main result, note that to our knowledge, mention of \eqref{eq:full_PDC}, as in \cite{Baccala06}, has not appeared in the literature with any explicit association with information theoretic ideas.

\section{RESULTS}
\label{sec:main_result}

\subsection{PDC}
\label{sec:pdc}
\begin{theorem}
\label{theo:main_result}
Let the $K$-variate wide sense stationary time series $\mathbf{x}(n)=[x_1(n) \ldots x_K(n)]^T$  satisfy \eqref{eq:model_law}, then
 \begin{equation} \label{eq:main_result_first}
 \iota\pi_{ij}(\omega) = \mbox{C}_{w_i\eta_j}(\omega),
 \end{equation}
where $\eta_j(n)=x_j(n)-\E[x_j(n)|\{x_l(m), \; l \neq j, \; m\in \Z\}]$ which is known as the partialized process associated to $x_j$  given the remaining time series.
 \end{theorem}
 
 \begin{cor} \label{cor:main}
Let  the $K$-variate Gaussian stationary time series $\mathbf{x}(n)=[x_1(n) \ldots x_K(n)]^T$  satisfy \eqref{eq:model_law}, then
 \begin{equation}
 \label{eq:main_result}
 \MIR(w_i, \eta_j)=-\frac{1}{4\pi}\int^{\pi}_{-\pi}\log (1-|\iota\pi_{ij}(\omega)|^2)d\omega.
 \end{equation}
 \end{cor}

\smallskip
\indent
To obtain the process $\eta_k$, remember that it constitutes the residue of the projection of $x_k$ onto the past, the future and the present of the remaining processes. Hence its autospectrum  is given by
\begin{equation}
 S_{\eta_k\eta_k}(\omega)=S_{x_kx_k}(\omega)
-\mathbf{s}_{x_k\mathbf{x}^k}(\omega)\mathbf{S}^{-1}_{\mathbf{x}^k\mathbf{x}^k}(\omega)\mathbf{s}_{\mathbf{x}^kx_k}(\omega),
 \end{equation}
 for $\mathbf{x}^k=[x_{l_1} \ldots x_{l_{K-1}}]^T$, $\{l_1,\ldots,l_{K-1}\}=\{1,\ldots,K\}\setminus \{k\}$ where
 $\mathbf{s}_{x_k \mathbf{x}^k}(\omega)$ is the $K-1$-dimensional vector whose entries are the cross spectra between $x_k$ and the remaining $K-1$ processes, whereas $\mathbf{S}_{\mathbf{x}^k\mathbf{x}^k}(\omega)$ is the spectral density matrix of $\mathbf{x}^k$. The spectrum $ S_{\eta_k\eta_k}(\omega)$ is also known in the literature as the partial spectrum of $x_k$ given $\mathbf{x}^k$ \citep{Priestley81}.

Note that 
\begin{equation}
\label{eq:wienerfilter}
\mathbf{g}_k(\omega)=\mathbf{s}_{x_k \mathbf{x}^k}(\omega)\mathbf{S}^{-1}_{\mathbf{x}^k\mathbf{x}^k}(\omega)
\end{equation}
constitutes an optimum Wiener filter whose role in producing $\eta_k$ is to deduct the influence of the other variables from $x_k$ to single out that contribution that is only its own.

Theorem \ref{theo:main_result} shows that PDC from $x_j$ to $x_i$ measures the amount of information common to the $\eta_j$ partial process  and the $w_i$ innovation. The proof is left to the Appendix but its main idea is to prove (\ref{eq:main_result_first}) so that (\ref{eq:main_result}) follows by use of (\ref{eq:gy_coherence}) to produce $\MIR(w_i,\eta_j)$.

\subsection{DTF}
\label{sec:dtf}

Every stationary process $\{\mathbf{x}(n)\}_{n \in \Z}$  with autoregressive representation (\ref{eq:model_law}) also has the following moving average representation
\begin{equation}
\label{eq:model_law2}
\mathbf{x}(n)=\sum_{l=0}^{+\infty}\mathbf{H}(l)\mathbf{w}(n-l),
\end{equation}
where the innovation process $\mathbf{w}$  is the same as that of (\ref{eq:model_law}).       

In connection to the  $h_{ij}(l)$ coefficients of $\mathbf{H}(l)$, consider the matrix
$\bar{\mathbf{H}}(\omega)$ with entries
\begin{equation}
\label{eq:mcoef}
\bar{H}_{ij}(\omega)=\sum\limits_{l=0}^{+\infty}h_{ij}(l)e^{-\mathbf{j}\omega l},
\end{equation}
and let $\bar{\mathbf{h}}_j(\omega)=\left[\bar{H}_{j1}(\omega) \dots  \bar{H}_{jK}(\omega) \right]^T$ whence follows
the definition of \textit{information} directed transfer function ($\iota$DTF) from $j$ to $i$ as
\begin{equation}
\label{eq:full_DTF}
\iota\gamma_{ij}(\omega)= 
\dfrac{\bar{H}_{ij}(\omega)\rho_{jj}^{1/2}}{\sqrt{{\bf \bar{h}}^H_j(\omega)%
\boldsymbol{\Sigma}_\mathbf{w}{\bf \bar{h}}_j(\omega)}},  
\end{equation}
where $\rho_{jj}$ is the variance of the partialized innovation process $\zeta_j(n)=w_j(n)-\E[w_j(n)/\{w_l(n), \; l \neq j\}]$ given explicitly by
\begin{equation*}
\rho_{jj}=\sigma_{jj}-\mathbf{\sigma}_{j\cdot}\boldsymbol{\Sigma}_{\cdot\cdot}^{-1}\mathbf{\sigma}_{j\cdot}^T,
\end{equation*}
where $\mathbf{\sigma}_{j\cdot}$ is the vector of covariances for innnovations \\
$\mathbf{w}^j(n)=[w_{l_1}(n) \ldots w_{l_{K-1}}(n)]^T$ where $\{l_{1}, \ldots, l_{K-1}\}=\{1,$ $\ldots, K\} \setminus \{j\}$ and $\boldsymbol{\Sigma}_{\cdot\cdot}$ is the covariance matrix of $\mathbf{w}^j(n)$.
 
When $\boldsymbol{\Sigma}_{\mathbf{w}}$ is the identity matrix, \eqref{eq:full_DTF} reduces to the original DTF from \cite{Kaminski91}. Also when $\boldsymbol{\Sigma}_{\mathbf{w}}$ is a diagonal matrix with distinct elements \eqref{eq:full_DTF}
reduces to directed coherence as defined in \cite{Baccala99}.

For this new quantity, a result analogous to Theorem \ref{theo:main_result} holds.

\begin{theorem}
\label{theo:main_result_DTF}
 Let  the $K$-variate wide sense stationary time series $\mathbf{x}(n)=[x_1(n) \ldots x_K(n)]^T$  satisfy \eqref{eq:model_law2},
 then
 \begin{equation} \label{eq:main_result_first_DTF}
 \iota\gamma_{ij}(\omega) = \mbox{C}_{x_i\zeta_j}(\omega),
 \end{equation}
where $\zeta_j$ is the previously defined partialized innovation process.
\end{theorem}

\begin{cor} \label{cor:mainDTF}
 Let  the $K$-variate Gaussian stationary time series $\mathbf{x}(n)=[x_1(n) \ldots x_K(n)]^T$  satisfy \eqref{eq:model_law2},
 then
 \begin{equation}
 \label{eq:main_result_DTF}
 \MIR(x_i, \zeta_j)=-\frac{1}{4\pi}\int^{\pi}_{-\pi}\log (1-|\iota\gamma_{ij}(\omega)|^2)d\omega.
 \end{equation}
\end{cor}

An important remark is that \eqref{eq:main_result_first}/\eqref{eq:main_result_first_DTF} hold for  wide-sense stationary processes
respectively with a autoregressive/moving average representations and that the gaussianity requirement is unnecessary for their validity.

Also the integrands in \eqref{eq:main_result} and \eqref{eq:main_result_DTF} are readily interpretable as
mutual information rates at each frequency.

\section{ILLUSTRATIVE EXAMPLE}
\label{sec:illustrations} 

Via the following simple accretive example it is possible to explicitly expose the nature of (\ref{eq:main_result_first}):
\begin{equation}
\label{eq:model_2}
\begin{bmatrix} x_1(n) \\ x_2(n) \end{bmatrix}=\begin{bmatrix} 0 & 0 \\ \alpha & 0 \end{bmatrix}\begin{bmatrix} x_1(n-1) \\ x_2(n-1) \end{bmatrix}+ \begin{bmatrix} w_1(n) \\ w_2(n) \end{bmatrix},
\end{equation}
where $\E[w_i(n)w_j(m)]=\delta_{nm}\delta_{ij}$, for $m,n \in \Z$ and $i,j \in \{1,2\}$ with $\delta_{pq}$ standing for the usual Kronecker delta symbol.

Clearly $\iota\pi_{12}(\omega)=0 $ and 
$$\iota\pi_{21}(\omega)=\dfrac{-\alpha e^{-\mathbf{j}\omega}}{\sqrt{1+\alpha^2}}.  $$

To obtain  $\mbox{C}_{w_1\eta_2}(\omega)$ using the fact that $$\mathbf{s}_{21}(\omega)\mathbf{S}^{-1}_{11}(\omega)=\alpha e^{-\mathbf{j}\omega}$$
implies $\eta_2(n)=x_2(n)-\alpha x_1(n-1) =w_2(n)$  so that
$\mbox{C}_{w_1\eta_2}(\omega)=0$, and hence $\iota\pi_{12}(\omega)=\mbox{C}_{w_1\eta_2}(\omega)$. 

Now to compute $\mbox{C}_{w_2\eta_1}(\omega)$ one must use the spectral density matrix of $[x_1\;\; x_2]^T$ given by
\begin{equation*}
\begin{bmatrix} S_{x_1x_1}(\omega) & S_{x_1x_2}(\omega) \\ S_{x_2x_1}(\omega) & S_{x_2x_2}(\omega)\end{bmatrix}=\begin{bmatrix} 1 & \alpha e^{\mathbf{j}\omega} \\ \alpha e^{-\mathbf{j}\omega}  & 1+\alpha^2 \end{bmatrix},
\end{equation*}
leading to the optimum filter $$\mathbf{G}_1(\omega)=\mathbf{s}_{12}(\omega)\mathbf{S}^{-1}_{22}(\omega)=\frac{\alpha}{1+\alpha^2}e^{\mathbf{j}\omega}$$ for $\E[x_1(n)/\{x_2(m), \; m \in \Z\}]$. It is noncausal and produces $\frac{\alpha}{1+\alpha^2}x_2(n+1)$ so that
\begin{equation*}
\eta_1(n)=x_1(n)-\frac{\alpha}{1+\alpha^2}x_2(n+1).
\end{equation*}

\noindent
Since $x_1(n)=w_1(n)$ and $x_2(n)=\alpha w_1(n-1)+w_2(n)$,
\begin{equation*}
\eta_1(n)=w_1(n)\frac{1}{1+\alpha^2}-w_2(n+1)\frac{\alpha}{1+\alpha^2},
\end{equation*}
which leads to
\begin{equation*}
S_{w_2\eta_1}(\omega)=\frac{-\alpha e^{-\mathbf{j}\omega}}{1+\alpha^2}
\end{equation*}
and
\begin{align*}
S_{\eta_1}(\omega)&=\frac{1}{1+\alpha^2}\\
S_{w_2}(\omega)&=1,
\end{align*}
showing that
\begin{equation*}
C_{w_2\eta_1}(\omega)=\frac{-\alpha e^{-\mathbf{j}\omega}}{\sqrt{1+\alpha^2}} 
\end{equation*}
confirms that $\iota\pi_{21}(\omega)=C_{w_2\eta_1}(\omega)$ via direct computation
of the Fourier transforms of the covariance/cross-covariance functions involving $w_2$ and $\eta_1$.

It is easy to verify that $\zeta_i(n)=w_i(n)$ so that direct computations
also confirm $\iota$PDC and $\iota$DTF equality in the $K=2$ case \citep{BaccaSame01} when 
$\boldsymbol{\Sigma}$ is the identity matrix.

Let model \eqref{eq:model_2} be enlarged by including a third observed variable
\begin{equation}
\label{eq:extra}
x_3(n)=\beta x_2(n-1)+w_3(n)
\end{equation}
where $w_3(n)$ is zero mean unit variance Gaussian and orthogonal to $w_1(n)$ and $w_2(n)$ for all lags.
This new equation means that the signal $x_1$ has an indirect path to $x_3$ via $x_2$
but no direct means of reaching $x_3$.

For this augmented model, the following joint moving average representation holds
\begin{align*}
&\begin{bmatrix} x_1(n) \\ x_2(n) \\ x_3(n)\end{bmatrix}=\begin{bmatrix} w_1(n) \\ w_2(n) \\ w_3(n) \end{bmatrix}+\begin{bmatrix} 0 & 0  & 0\\ \alpha & 0 & 0\\ 0 & \beta & 0 \end{bmatrix}\begin{bmatrix} w_1(n-1) \\ w_2(n-1) \\ w_3(n-1) \end{bmatrix}\\
&+\begin{bmatrix} 0 & 0  & 0\\ 0 & 0 & 0\\ \alpha \beta & 0 & 0 \end{bmatrix}\begin{bmatrix} w_1(n-2) \\ w_2(n-2) \\ w_3(n-2) \end{bmatrix},
\end{align*}
which produces
\begin{align}
\iota\gamma_{21}(\omega)&=\frac{\alpha e^{-\mathbf{j}\omega} }{\sqrt{1+\alpha^2}},\nonumber\\
\iota\gamma_{32}(\omega)&=\frac{\beta e^{-\mathbf{j}\omega} }{\sqrt{1+\beta^2+\alpha^2\beta^2}},\nonumber\\
\iota\gamma_{31}(\omega)&=\frac{\alpha \beta e^{-2\mathbf{j}\omega} }{\sqrt{1+\beta^2+\alpha^2\beta^2}},\label{DTF_31}\\
\nonumber
\end{align}
and $\iota\gamma_{kl}=0$ for $l > k$ by direct computation using \eqref{eq:full_DTF}. To verify
\eqref{eq:main_result_first_DTF}, one obtains $\zeta_i=w_i$ since the $w_i$ innovations are uncorrelated leading to
\begin{align*}
&S_{x_2\zeta_1}(\omega)=\alpha,\;\;\;\;\;S_{x_3\zeta_2}(\omega)=\beta e^{-\mathbf{j}\omega},\\
&S_{x_3\zeta_1}(\omega)=\alpha \beta e^{-2\mathbf{j}\omega},\;\;\;\; S_{x_2x_2}(\omega)=1+\alpha^2,\\
&S_{x_3x_3}(\omega)=1+\beta^2+\alpha^2\beta^2,\;\;S_{\zeta_1\zeta_1}(\omega)=1=S_{\zeta_2\zeta_2}(\omega),\\
\end{align*}
wherefrom $\iota\gamma_{21}(\omega)=\mbox{C}_{x_2\zeta_1}(\omega)$, $\iota\gamma_{32}(\omega)=\mbox{C}_{x_3\zeta_2}(\omega)$ and $\iota\gamma_{31}(\omega)=\mbox{C}_{x_3\zeta_1}(\omega)$ using \eqref{eq:main_result_first_DTF}.

One may compute this model's PDCs
\begin{align*}
\iota\pi_{21}(\omega)&=\frac{-\alpha e^{-\mathbf{j}\omega} }{\sqrt{1+\alpha^2}},\\
\iota\pi_{32}(\omega)&=\frac{-\beta e^{-\mathbf{j}\omega} }{\sqrt{1+\beta^2}},\\
\iota\pi_{31}(\omega)&=0,
\end{align*}
either via \eqref{eq:full_PDC}, or via Theorem \ref{theo:main_result}.

This exposes the fact that the augmented model's
direct interaction is represented by PDC whereas DTF from $x_1$ to $x_3$  \eqref{DTF_31}
is zero if either $\alpha$ or $\beta$ is zero. This means that a signal pathway leaving $x_1$ reaches $x_3$ so that DTF therefore represents the net directed effect of $x_1$ onto $x_3$ as in fact previously noted in \cite{baccalapbr}.

\section{Discussion}
\label{sec:discussion}

In their \textit{information} forms both PDC and DTF represent true coherences and thus constitute
complete alternative descriptions of the dynamic relations involving the observed vector time series $\mathbf{x}(n)$
and $\mathbf{w}(n)$, the innovations vector process, or its orthogonalized version $\boldsymbol{\zeta}(n)$, which summarize the stochastic novelty after stripping all mutual
correlations present in $\mathbf{x}(n)$.

DTF can be thought of as a forward description for it depicts how $\boldsymbol{\zeta}(n)$ affect the  $\mathbf{x}(n)$ observations, whereas PDC describes how $\mathbf{w}(n)$ relates to $\boldsymbol{\eta}(n)$ which is obtained by the mutual partialization of the components of $\mathbf{x}(n)$. Thus, $\eta_i(n)$ essentially excludes
those redundancies in $x_i(n)$ that can be attributed to the other $x_j(n)$ ($j\neq i$).
This redundancy extraction, by direct analogy with linear algebraic procedures gives rise to $\{\eta_j, j=1,\dots, K\}$ as a \textit{dual} (also frequently termed \textit{reciprocal}) basis to the $\{x_i, i=1,\dots, K\}$ basis as its elements $\eta_i$ are orthogonal to all $x_j$ for $i\neq j$. It is in this precise sense that PDC's description is dual to DTF's - they map the innovations onto dual representations of the observed dynamics.

A question that may come to mind is: how can DTF (PDC) being related to mutual information, a recognizedly reciprocal quantity, are able to describe unreciprocal aspects of the interaction between time series? The answer lies in that they relate the
$x_i$ ($\eta_i$) to innovations $\zeta_j$ ($w_j$) so that permuting $i$ and $j$ describes the relationship between distinct inner component subprocesses. As such, for example, in the case of PDC, $\MIR(w_i,\eta_j)$ (for $|\iota \pi_{ij}(\omega)|^2$) and $\MIR(w_j,\eta_i)$ (for $|\iota \pi_{ji}(\omega)|^2$) are not equal in general as opposed to  
\begin{equation}
\label{eq:mir_equal}
\MIR(w_i,\eta_j)=\MIR(\eta_j,w_i)
\end{equation}
whose equality always holds because $\ |\iota \pi_{ij}(\omega)|^2=|\iota \pi^*_{ij}(\omega)|^2$ where $^*$ denotes complex conjugation. In other words, index permutation in PDC entails comparing different underlying intrinsic component processes. A similar result holds for DTF.

Another point is why PDC/DTF are related to Gran\-ger causality. This is so because the inherent decorrelation \\$\E\left[w_i(n)w_j(m)\right]=0$ for all $i,j$ provided that $n\neq m$ introduces the necessary time asymetry to allow
their causal interpretations. Also observe that by definition of innovation, time asymetry is an automatic consiquence of
$w_i(n)$'s uncorrelation to $\eta_j(k)$ for $k<n$. The same holds for $\zeta_j(n)$ which is uncorrelated with
$x_i(k)$ for $k \leq n$.

Though left to the Appendix, the proof of Theorem \ref{theo:main_result} reveals an interesting aspect, namely eq. \eqref{eq:totprove}
that allows interpreting $\bar{A}_{ij}(\omega)$ \eqref{eq:acoef} as a transfer function from $\eta_j$ to $w_i$. This observation sheds light on  \cite{Schelter2009121}'s employment of a studentized version of $\bar{A}_{ij}(\omega)$ in characterizing the relationship between $\mathbf{x}(n)$ components. Similar observations hold for the  $\bar{H}_{ij}(\omega)$, whose 
magnitude has been used by \cite{Blinowska2010205}.

PDC and DTF are not alone as attempts to describe information flow between multivariate time series.To discuss these ideas one must also mention the efforts of \cite{Geweke82} and \cite{Hosoya91}. Though delving into detailed  and specific comparative aspects of their proposals vis-à-vis those described herein is beyond our intended scope and is plan\-ned for future publications, it is perhaps reassuring to note when just time series pairs are considered ($K=2$) all of the latter frequency domain measures coalesce into one and the same measure.

As a matter of fact, for $K=2$, it is possible to show that 
\begin{equation}
|\iota\pi_{ij}(\omega)|^2 = |\iota\gamma_{ij}(\omega)|^2=
1-e^{-\mbox{f}_{j \to i}(\omega)}=1-e^{-\mbox{M}_{j \to i}(\omega)}
\end{equation}
for ($i,j \in \{1,2\}$) where $\mbox{f}_{j \to i}(\omega)$ and $\mbox{M}_{j \to i}(\omega)$ describe respectively Ge\-we\-ke's and Hosoya's frequency domain causal measures in their own notation (the arrow shows the direction information flow). Furthermore, when it comes to testing for the null hypothesis of Granger causality when $K=2$, it is straightforward to verify the equivalence of the following statements:
\begin{enumerate}
\item There is no Granger causality from $x_j$ to $x_i$.
\item $\MIR(x_i, \zeta_j) = 0$.
\item $\MIR(w_i, \eta_j) = 0$.
\item $|\iota\pi_{ij}(\omega)|^2 = 0, \; \forall \omega \in [-\pi,\pi)$.
\item $|\iota\gamma_{ij}(\omega)|^2 = 0, \; \forall \omega \in [-\pi,\pi)$.
\item $\mbox{f}_{j \to i}(\omega) = 0, \; \forall \omega \in [-\pi,\pi)$.
\item $\mbox{M}_{j \to i}(\omega) = 0, \; \forall \omega \in [-\pi,\pi)$.
\item $\bar{A}_{ij}(\omega)=0, \; \forall \omega \in [-\pi,\pi)$.
\item $H_{ij}(\omega)=0, \; \forall \omega \in [-\pi,\pi)$.
\end{enumerate}

Which of the above statements is more convenient depends on criteria like knowledge of precise asymptotic statistics and test power. In fact, precise results this kind for the general $K>2$ case, that also include asymptotic
confidence intervals, are known for $|\iota\pi_{ij}(\omega)|^2$ and are being prepared for submission.

Though time domain considerations are strictly outside our scope, they are required to fully understand the
difference between the various measures \citep{Takahashi} and underlie the difficulties of generalizing Gewe\-ke's and
Hosoya's proposals to $K>2$ as attempted respectively in \cite{Geweke84} and \cite{Hosoya01} while keeping a consistent 
interpretation of information flow in association with Granger causality.

A summary of the relationships between the underlying processes addressed in this paper is portrayed in Figure \ref{fig:ships}.

\begin{figure}
\begin{center}
\includegraphics[width=10cm]{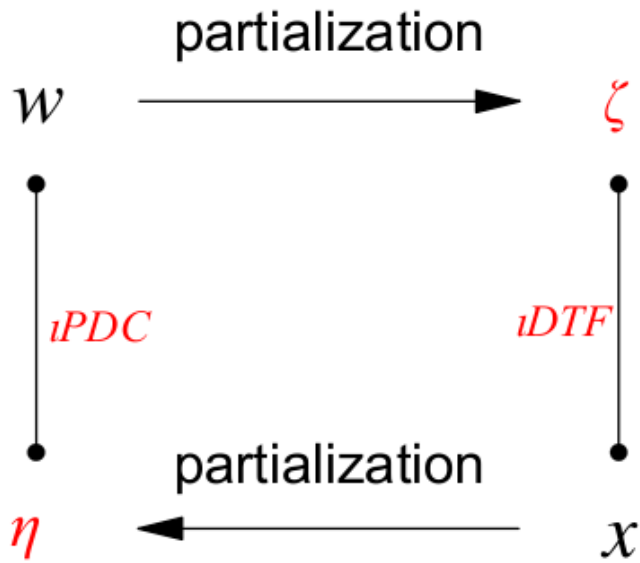}
\caption{The diagram summarizes the relationship between the various descriptive 
processes  associated with the original observations $\mathbf{x}$ including the
innovations process $\mathbf{w}$ and their respective partialized versions $\boldsymbol{\eta}$ and
$\boldsymbol{\zeta}$. The mutual information rate relationships are described by $\iota PDC$ for ($\boldsymbol{\eta}$,$\mathbf{w}$) and $\iota DTF$ for ($\boldsymbol{\zeta}$,$\mathbf{x}$) . Information
flow sources are indexed by the greek type face vector components and the information receiving structures are chosen
among the components of the latin type face processes.}
\label{fig:ships}
\end{center}
\end{figure}

Finally, it should be noted that iPDC, as herein defined, provides
an absolute signal scale invariant measure of direct connectivity strength
between observed time series as opposed to either PDC or gPDC that provide
only relative coupling assessments.

\section{Conclusion}
\label{sec:conclusion}

New properly weighted multivariate directed dependen\-ce meas\-ures between stochastic processes that generalize PDC and DTF 
have been introduced and their relationship to mutual information has been spelled out in terms of more fundamental
adequately partialized processes. These results enlighten the relationship of formerly available connectivity meas\-ures
and the notion of information flow. Theorem \ref{theo:main_result} is a novel result. For bivariate time series, results similar to Theorem  \ref{theo:main_result_DTF} have appeared several times in the literature in association with Geweke's measure of directed dependence \cite{Geweke82}. 
The $\iota$DTF introduced herein is novel and constitutes a proper generalization of Gewe\-ke's result for the multivariate setting while $\iota$PDC´s result (also novel) is its dual.

The present results not only introduce a unified fra\-me\-work to understand connectivity measures, but also open new generalization perspectives in nonlinear interaction cases for which information theory seems to be the natural study toolset.

\begin{appendix}

\section{Appendix}
\label{sec:proofs}
\subsection{Proof of Theorem \ref{theo:main_result} and Corollary \ref{cor:main}} \label{sec:prooftheo1}

Before proving Theorem \ref{theo:main_result} consider the following lemma:

\begin{lemma} \label{lemma:inv}
Let $\mathbf{S}_{\mathbf{x}}(\omega)$ be the power spectral density matrix of the stationary $K$-variate time series $\mathbf{x}(n)$ obeying \eqref{eq:model_law}. Then
\begin{equation}
\label{eq:inverse_partial_spec}
\mathbf{\bar{a}}_j(\omega)^H\mathbf{\Sigma}_\mathbf{w}^{-1}\mathbf{\bar{a}}_j(\omega) = \mathbf{S}_{\eta_{j}\eta_{j}}^{-1}(\omega)
\end{equation}
holds.
\end{lemma}

\begin{proof}
One may rewrite \eqref{eq:acoef} in matrix form as
\begin{equation}
\mathbf{\bar{A}}(\omega)=\mathbf{I}-\sum_{l=1}^{+\infty}\mathbf{A}(l)e^{-\mathbf{j}\omega l},
\end{equation}
where $\mathbf{I}$ is the $K\times K$identity matrix.

Because \eqref{eq:model_law} holds,
the inverse of the $\mathbf{x}$ may be written as
\begin{equation}\label{eq:comp}
\mathbf{S}^{-1}_\mathbf{x}(\omega)=\mathbf{\bar{A}}(\omega)^H\mathbf{\Sigma}_\mathbf{w}^{-1}\mathbf{\bar{A}}(\omega).
\end{equation}
each of whose elements $[\cdot]_{jj}$
\begin{equation*}
\label{eq:inv_formula}
\left[\mathbf{S}^{-1}_\mathbf{x}(\omega)\right]_{jj}=\left(S_{x_jx_j}(\omega)
-\mathbf{s}_{x_j\mathbf{x}^j}(\omega)\mathbf{S}^{-1}_{\mathbf{x}^j\mathbf{x}^j}(\omega)\mathbf{s}_{\mathbf{x}^jx_j}(\omega)\right)^{-1},
\end{equation*}
follows from the partitioned matrix inversion formula (see \textit{e. g.} the appendix in \citet{Lutkepohl93}) which equals
\begin{equation}
\label{eq:direct_comp}
S_{\eta_{j}\eta_{j}}^{-1}(\omega)=\mathbf{\bar{a}}^H_j(\omega)\mathbf{\Sigma}_\mathbf{w}^{-1}\mathbf{\bar{a}}_j(\omega).
\end{equation}
by direct computation of $\left[\mathbf{S}_\mathbf{x}(\omega)^{-1}\right]_{jj}$ from \eqref{eq:comp}.
\end{proof}

Lemma \ref{lemma:inv} allows rewriting \eqref{eq:full_PDC} as
\begin{equation*}
\iota\pi_{ij}(\omega)= \bar{A}_{ij}(\omega)\sigma_{ii}^{-1/2}\sqrt{S_{\eta_{j}\eta_{j}}(\omega)}.
\end{equation*}

Hence to prove Theorem \ref{theo:main_result} all one must show is that
\begin{equation}\label{eq:equlproof}
\mbox{C}_{w_i\eta_j}(\omega)\buildrel\triangle\over =\frac{S_{w_i\eta_j}(\omega)}{\sqrt{S_{w_iw_i}(\omega)S_{\eta_j\eta_j}(\omega)}} = \bar{A}_{ij}(\omega)\sigma_{ii}^{-1/2}\sqrt{S_{\eta_{j}\eta_{j}}(\omega)}.
\end{equation}

Since
\begin{equation*}
S_{w_iw_i}(\omega)=\sigma_{ii},
\end{equation*}
proving \eqref{eq:equlproof} reduces to showing
\begin{equation} \label{eq:totprove}
 \bar{A}_{ij}(\omega) = \frac{S_{w_i\eta_j}(\omega)}{S_{\eta_j\eta_j}(\omega)}.
\end{equation}

By straightforward computation with help of \eqref{eq:wienerfilter}
\begin{equation*}
S_{w_i\eta_j}(\omega) = \sum_{l=1}^{K}\bar{A}_{il}(\omega)\left\{S_{x_lx_j}(\omega)-\mathbf{s}_{x_j \mathbf{x}^j}(\omega)\mathbf{S}^{-1}_{\mathbf{x}^j\mathbf{x}^j}(\omega)\mathbf{s}_{x_l\mathbf{x}^j}(\omega)\right\}.
\end{equation*}
whose right-hand side can be broken as
\begin{align} \label{eq:part1}
& \bar{A}_{ij}(\omega)S_{\eta_j\eta_j}(\omega) + \nonumber\\ 
& \sum_{\substack{l=1 \\ l\neq j}}^{K}\bar{A}_{il}(\omega)\left\{S_{x_lx_j}(\omega)-\mathbf{s}_{x_j \mathbf{x}^j}(\omega)\mathbf{S}^{-1}_{\mathbf{x}^j\mathbf{x}^j}(\omega)\mathbf{s}_{x_l\mathbf{x}^j}(\omega)\right\}.
\end{align}
which simplifies to
\begin{equation*}
S_{w_i\eta_j}(\omega) = \bar{A}_{ij}(\omega)S_{\eta_j\eta_j}(\omega)
\end{equation*}
as the partialized process $\eta_j$ is orthogonal to $\mathbf{x}^j$ by construction, i.e.
\begin{equation} \label{eq:part2}
S_{x_lx_j}(\omega)-\mathbf{s}_{x_j \mathbf{x}^j}(\omega)\mathbf{S}^{-1}_{\mathbf{x}^j\mathbf{x}^j}(\omega)\mathbf{s}_{x_l\mathbf{x}^j}(\omega) = 0, \;\; \forall \; l \neq j,\;\;  \omega \in [-\pi,\pi).
\end{equation}
thereby concluding the proof.

When $w_i$ and $\eta_j$ are stationary Gaussian, Corollary \ref{cor:main} is a direct consequence of applying the following
\begin{theorem}[Gelfand and Yaglom,\citeyear{Gelfand59}] \label{theo:gelfand}
Let $x$ and $y$ be jointly Gaussian stationary time series. Assume that $$\E[w_x(n)w_y(n)]^2 < \E[w_x^2(n)] \E[w_y^2(n)],$$
where $w_x(n)$ and $w_y(n)$ are the innovations associated to $x$ and $y$. Then the following equality holds:
\begin{equation}
\MIR(x,y)=-\frac{1}{4\pi}\int^{\pi}_{-\pi}\log (1-|\mbox{C}_{xy}(\omega)|^2)d\omega,
\end{equation}
when $w_i$ and $\eta_j$ are both jointly stationary Gaussian.
\end{theorem}

\subsection{Proof of Theorem \ref{theo:main_result_DTF} and Corollary \ref{cor:mainDTF}} \label{sec:prooftheo2}

To prove Theorem \ref{theo:main_result_DTF} recall that by definition
\begin{equation*}
C_{x_i\zeta_j}(\omega)=\frac{S_{x_i\zeta_j}(\omega)}{\sqrt{S_{x_ix_i}(\omega)S_{\zeta_j\zeta_j}(\omega)}}.
\end{equation*}
and that
\begin{equation*}
S_{\zeta_j\zeta_j}(\omega)=\rho_{jj}.
\end{equation*}

Therefore, it suffices to show that
\begin{equation} 
\dfrac{\bar{H}_{ij}(\omega)}{\sqrt{{\bf \bar{h}}^H_j(\omega)
\boldsymbol{\Sigma}_\mathbf{w}{\bf \bar{h}}_j(\omega)}} = \frac{S_{x_i\zeta_j}(\omega)}{\sqrt{S_{x_ix_i}(\omega)}}.
\end{equation}

By the existence of the moving average representation \eqref{eq:model_law2}
\begin{equation}
S_{x_ix_i}(\omega)= {\bf \bar{h}}^H_j(\omega)
\boldsymbol{\Sigma}_\mathbf{w}{\bf \bar{h}}_j(\omega).
\end{equation}

Also, by the existence of moving average representation \eqref{eq:model_law2} and the orthogonality of the partialized innovation process $\zeta_j$ with respect to the innovations $w_l, \; l\neq j$, it follows that
\begin{equation*}
S_{x_i\zeta_j}(\omega) = \bar{H}_{ij}(\omega).
\end{equation*}
and this concludes the proof.

Corollary \ref{cor:mainDTF} follows immediately from Theorems  \ref{theo:main_result_DTF} and \ref{theo:gelfand}.

\end{appendix}

\section{ACKNOWLEDGMENTS}

The authors gratefully acknowledge support from the FAPESP/CINAPCE 2005/56464-9 Grant.
D.Y.T. to CAPES Grant and FAPESP Grant 2008/08171-0.
L.A.B  to CNPq Grants 306964/2006-6 and 304404/2009-8 
K.S. to CNPq Grant

\bibliographystyle{plainnat} 
\bibliography{ba} 

\end{document}